\documentclass[12pt,a4paper, reqno]{amsart}
\usepackage{amsfonts,amsmath,amsthm,amssymb, amscd}
\usepackage[top=2cm, bottom=5cm, left=3cm, right=1cm]{geometry}
\def\gp#1{\langle#1\rangle}

\usepackage[active]{srcltx}

\newtheorem{theorem}{Theorem}
\newtheorem{lemma}{Lemma}
\newtheorem{corollary}{Corollary}

\title[Locally nilpotent unit group]{Group algebra whose  unit group \\ is locally nilpotent}
\author{V.~Bovdi}

\dedicatory{Dedicated to the memory of Professor L.G.~Kov\'acs}
\thanks{
Supported by the UAEU grants:  UPAR  G00001922}

\address{Department of Math. Sciences\\
UAE University\\
Al-Ain\\
United Arab Emirates}
\email{vbovdi@gmail.com}

\begin{document}
\maketitle

\begin{abstract}
We present a complete  list of   groups $G$ and  fields $F$ for which: (i) the group of normalized units $V(FG)$ of the group algebra $FG$ is  locally nilpotent;
(ii) the   group  algebra $FG$  has    a finite number of nilpotent elements  and   $V(FG)$  is an Engel group.
\end{abstract}

\section{Introduction}
Let $V(FG)$ be the normalized subgroup of the group of units $U(FG)$ of the group algebra $FG$  of a group $G$ over a field $F$ of characteristic ${\rm char}(F)=p\geq 0$. It is well known that $U(FG)=V(FG)\times U(F)$, where $U(F)=F\setminus \{0\}$.
The group of normalized units $V(FG)$ of a modular group algebra $FG$ has a complicated  structure and was  studied  in several papers. For an overview we  recommend the survey paper \cite{Bovdi_survey}.

An explicit list of groups $G$ and rings $K$ for which $V(KG)$ are nilpotent was  obtained by I.~Khripta (see \cite{Khripta_1} for the modular case and \cite{Khripta_2} for the non-modular case). In \cite{Bovdi_solvable} it was completely determined when   $V(FG)$ is solvable.
It is still a challenging problem  whether $V(FG)$ is an  Engel group. This question has a long history (see \cite{Bovdi_engel, Bovdi_Khripta_1, Bovdi_Khripta_2,Bovdi_Khripta_4, Bovdi_Khripta_3, Shalev_I}). The non-modular case was solved  by A.~Bovdi (see \cite{Bovdi_engel}, Theorem 1.1, p.174). For the modular case there is no complete solution (see \cite{Bovdi_engel}, Theorem 3.2, p.175),  but there is a full   description of $FG$ when $V(FG)$ is a bounded Engel group (see \cite{Bovdi_engel}, Theorem 3.3, p.176).

It is well known (for example, see \cite{Traustason}) that the Engel property of a group is close to its  local nilpotency. A locally nilpotent group  is always   Engel (see \cite{Traustason}). However these classes of groups do not coincide (see  the Golod's counterexample in \cite{Golod}). The following results  are classical (see \cite{Traustason}): each  Engel profinite group (see  \cite{Wilson_Zelmanov}), each compact Engel group (see \cite{Medvedev}),  each Engel linear group (see \cite{Suprunenko_Garashchuk}), each 3-Engel and 4-Engel group  (see \cite{Heineken} and  \cite{Havas_Vaughan-Lee})  and  all Engel groups satisfying max (see \cite{Baer}) are   locally nilpotent.

A group $G$ is said to be {\it Engel } if for any $x,y\in G$ the equation
$(x,y,y, \ldots,y)=1$ holds, where $y$  is repeated in the commutator sufficiently many times depending on $x$ and $y$. We shall use the left-normed simple commutator notation
$(x_1,x_2)=x_1^{-1}x_2^{-1}x_1x_2$\quad  and
\[
(x_1,\ldots,x_{n})=\big((x_1,\ldots,x_{n-1}),x_n\big), \qquad\quad  (x_{1},\ldots, x_n\in G).
\]
A group is called {\it locally nilpotent} if  all its  f.~g. (finitely generated) subgroups are nilpotent. Such  a group is always   Engel (see \cite{Traustason}). The set  of  elements of finite orders of a group $G$ (which is not necessarily a subgroup) is called {\it the torsion part} of $G$ and is denoted   by $\mathfrak{t}(G)$. We use the notion and results from  the book \cite{Bovdi_book} and the survey papers \cite{Bovdi_survey, Traustason}.

In several articles,      M.~Ramezan-Nassab attempted   to   describe   the structure of  groups $G$ for which $V(FG)$ are Engel (locally nilpotent) groups in the case when  $FG$ have only a finite number of nilpotent elements (see Theorem 1.5 in \cite{Ramezan-Nassab_1}, Theorems 1.2 and 1.3 in  \cite{Ramezan-Nassab_2} and Theorem 1.3 in  \cite{Ramezan-Nassab_3}).  The following  theorem   gives  a complete answer.

\begin{theorem}\label{T:1}
Let $FG$  be the  group algebra of a group $G$.   If  $FG$  has   only a  finite number of non-zero nilpotent elements,   then $F$ is a finite field  of ${\rm char}(F)=p$. Additionally, if    $V(FG)$ is an Engel group,  then $V(FG)$ is nilpotent,  $G$ is a finite group such that  $G=Syl_p(G)\times A$, where $Syl_p(G)\not=\gp{1}$, $G'\leq Syl_p(G)$ and  $A$ is a central subgroup of $G$.
\end{theorem}


The next result plays a technical role.

\begin{theorem}\label{T:2}
Let $G$ be a group such that  $G'$ is a locally finite p-group and  the  $p$-Sylow subgroup  $P=Syl_q(G)$ of $G$ is normal in  $G$.
If $V(FG)$ is an Engel group and ${\rm char}(F)=p>0$ then the following conditions hold:
\begin{itemize}
\item[({i})] $V'\leq 1+\mathfrak I (G')$ and $1+\mathfrak I (P)=Syl_p( V(FG))$;

\item[(ii)]  for each $q\not=p$, the group  $Syl_q(G)$ is  central;

\item[(iii)]  the group $\mathfrak{t}(G)=P\times D$,
where $D= \times_{\substack{q\not=p}} Syl_q(G)$;
\item[(iv)] let $M$ be the subgroup of $G/P$ generated by the central subgroup $DP/P$ and by all the elements of infinite order in $G/P$. Then  $V(FG)/(1+\mathfrak I (P))\cong V(FM)$    and    $V(FM)/V(FD)$   is a  torsion-free group;
\item[(v)] $G/\mathfrak{t}(G)$ is an abelian torsion-free group and, if $D$ is  finite, then
\begin{equation}\label{E:1}
\begin{split}
V(FM)\cong   V(FD)  \times (\underbrace{G/\mathfrak{t}(G)\times\cdots\times  G/\mathfrak{t}(G)}_n),\\
\end{split}
\end{equation}
where $n$ is the  number of summands  in the  decomposition of $FD$ into a direct sum of fields.
\end{itemize}
\end{theorem}


The next two theorems completely  describe  groups $G$  with  $V(FG)$   locally nilpotent. Some special  cases  of   Theorem 3  were  proved by I.~Khripta (see \cite{Khripta_2}) and M.~Ramezan-Nassab (see Theorem 1.2 in \cite{Ramezan-Nassab_2} and Corollary 1.3 and Theorem 1.4 in \cite{Ramezan-Nassab_1}).

\begin{theorem}\label{T:3}
Let $FG$ be a modular group algebra  of a group $G$ over the field $F$  of positive characteristic $p$.
The group $V(FG)$ is locally nilpotent if and only if $G$ is locally nilpotent and
 $G'$ is a $p$-group.
\end{theorem}


\begin{theorem}\label{T:4}
Let $FG$ be a non-modular group algebra   of characteristic $p\geq 0$.
The group $V(FG)$ is  locally nilpotent if and only if $G$ is a locally nilpotent group,  $\mathfrak{t}(G)$  is an abelian group  and one of the following conditions holds:
\begin{itemize}
\item[(i)]  $\mathfrak{t}(G)$  is a central subgroup;

\item[(ii)]   $F$ is a prime field of characteristic $p=2^t-1$,\quad the exponent  of $\mathfrak{t}(G)$
divides $p^2-1$ and   $g^{-1}ag=a^p$ for all $a\in \mathfrak{t}(G)$ and  $g\in G\setminus C_G(\mathfrak{t}(G))$.
\end{itemize}
\end{theorem}

As a  consequence of  Theorems \ref{T:2} and \ref{T:3} we obtain the classical result of I.~Khripta.

\begin{corollary}\label{C:1}
Let $FG$ be a modular group algebra   of  positive characteristic $p$. The group
$V(FG)$ is  nilpotent if and only if $G$ is nilpotent and $G'$  is a finite $p$-group.

The structure of $V(FG)$  is the following: the group $\mathfrak{t}(G)=P\times D$, where  $P$ is the   $p$-Sylow subgroup   of $G$,  $D$ is a central subgroup,  $1+\mathfrak I (P)$ is the $p$-Sylow subgroup of $V(FG)$ and $V'\leq 1+\mathfrak I (G')$. Moreover, if $D$ is a finite abelian group, then we have the following isomorphism between abelian groups
\[
\begin{split}
V(FG)/\big(1+\mathfrak{I}(P)\big) \cong  V(FD)  \times (\underbrace{G/\mathfrak{t}(G)\times\cdots\times  G/\mathfrak{t}(G)}_n),\\
\end{split}
\]
where $n$ is the  number of summands  in the  decomposition of $FD$ into a direct sum of fields.
\end{corollary}

\section{Proofs}
We start our proof with the following important

\begin{lemma}\label{L:1}
Let $FG$ be the group algebra of a group $G$  such that    $V(FG)$ is an  Engel group. If  $H$ is  a subgroup  of  $\mathfrak{t}(G)$ and  $H$
does not contain an  element of order  ${\rm char}(F)$,  then  $H$ is abelian, each subgroup of $H$ is normal in $G$ and each  idempotent of $FH$ is central in $FG$.
\end{lemma}

\begin{proof} If $a\in \mathfrak{t}(G)$ and $g \in G\setminus N_{G}(\gp{a})$, then  $x=\widehat{a}g(a-1)\not=0$  and   $1 + x\in V(FG)$. A straightforward calculation shows that
\[
(1+x, a, m) =1+x(a -1)^m, \qquad (m\geq 1).
\]
Since $V(FG)$ is an Engel group,  $x(a-1)^s\ne 0$ and $x(a -1)^{s+1}=0$ for a suitable $s\in \mathbb{N}$. It follows
that $x(a -1 )^sa^i =x(a-1)^s$ for each  $i\geq 0$ and
\[
x(a-1)^s\cdot |a|=x(a-1)^s(1+a+a^2+\cdots +a^{|a|-1})=0.
\]
Since   ${\rm char}(F)$ does   not divide $|a|$, we have that
$x(a -1 )^s=0$, a contradiction. Hence  every finite cyclic subgroup of $H$ is
normal in $G$, so  $H$ is either abelian or hamiltonian.

If $H$ is  a hamiltonian group,  then using the same proof as in the second part of
Lemma 1.1 of  \cite{Bovdi_Khripta_3} (see  p.122), we obtain a contradiction.

Hence  $H$ is  abelian and each of its subgroups is normal in $G$.

We claim that all idempotents of $LT$ are central  in  $L\gp{T,g}$, where $T$ is  a  finite abelian subgroup of $H$, $g\in G$ and $L$ is   a prime  subfield of the field $F$.

Assume on the contrary that there exist  a primitive idempotent $e\in LT$ and $g\in G$ such that $geg^{-1}\ne e$. The element  $b=g^{-1}a^{-1}ga \neq 1$ for some $a \in T$ and
\[
W=\gp{c^{-1}(g^{-1}cg)\mid c\in T}=\gp{T,g}'\vartriangleleft \gp{T,g}
\]
is a  non-trivial finite abelian subgroup of $T$.

Obviously each subgroup of $T\leq H$ is normal in $\gp{T,g}$,  so the idempotent $f= \frac{1}{|W|}\widehat{\gp{W}}\in L[T]$ is  central in $L\gp{T, g}$ and can be expressed as $f=f_1+\cdots+f_s$ in which  $f_1,\ldots, f_s$ are  primitive and  mutually
orthogonal idempotents of the  finite dimensional semisimple algebra $L[T]$.

The idempotent    $e$ does  not appear in the decomposition of $f$.  Indeed,   otherwise we have $ef=e$. If $e=\sum_i\alpha_it_i$, where $\alpha_i\in L$ and $t_i\in T$, then
\[
g^{-1}(ef)g=g^{-1}egf=\sum_i\alpha_ig^{-1}t_igf=\sum_i\alpha_it_i (t_i, g)f=ef
\]
so $e=fe$ is central, a contradiction.

If $ef\not=0$, then $ef=e$, again a contradiction. Thus  $ef=0$.

Consider  $f_*= \frac{1}{|b|}\widehat{\gp{b}}\in L[T]$. Since $\gp{b}\subseteq W$, the idempotent
$f_*$    appears in the decomposition of $f$, so   $ef_*=0$. Furthermore    $geg^{-1}\in L[T]$ is primitive as an automorphism's  image  of the primitive idempotent $e$, so $ege=e(geg^{-1})g=0$.
Evidently  $(1 + eg)^{-1}=1- eg$ and
\[
\begin{split}
(1+eg,a)&=(1-eg)(1+ea^{-1}ga)\\&=(1-eg)(1+egb)=1+eg(b-1).
\end{split}
\]
Now an easy induction on $n\geq 1$ shows that
$(1 +eg,a,n) =1+eg(b -1)^n$.  However  $V(FG)$ is Engel, so there exists $m$ such that
\[
eg(b - 1)^m\ne 0\qquad  \text{and} \qquad  eg(b -1)^{m+1}= 0.
\]
Thus $eg(b-1)^mb^i = eg(b -1)^m$ for any $i\geq 0$, which leads to the
contradiction
\[
\begin{split}
eg(b -1)^m = eg(b -1)^m f_*= eg(b -1)^{m-1}\big((b-1) f_*\big)= 0.
\end{split}
\]
Therefore, all idempotents of $FH$ are central  in  $FG$.
\end{proof}

\smallskip

\begin{proof}[\underline{Proof of Theorem \ref{T:1}}]
Let $N_*$ be the (finite) set of all nilpotent elements of $FG\setminus 0$. If $x\in N_*$  then  the subset of nilpotent elements $\{\lambda x\mid \lambda\in F\}$ is finite, so $F$ is a finite field of ${\rm char}(F)=p$.

Assume that there exist $u\in N_*$ and $g\in G\setminus \mathfrak{t}(G)$  such that $gu=ug$. Since the right annihilator $L$ of the left ideal $\frak{I}_l(\gp{g})=\gp{g^i-1\mid i\in \mathbb{Z}}_{FG}$
is different from zero  if and only if $\gp{g}$ is a finite group (see Proposition 2.7, \cite{Bovdi_book}, p.9), the set
$\{ u(g^i-1)\mid i\in \mathbb{Z}\}$ is an infinite subset of $N_*$, a contradiction.

Let $0\not=u\in N_*$. Clearly  $g^{-i}ug^i\not=0$ and  $g^{-i}ug^i\in N_*$ for any $g\in G\setminus \mathfrak{t}(G)$ and $i\in \mathbb{Z}$. Since $N_*$ is a finite set,  at least for one $i\in \mathbb{Z}$ there exists  $j\in \mathbb{Z}$  ($j\not=i$) such that $g^{-i}ug^i=g^{-j}ug^j$, so $g^{i-j}u=ug^{i-j}$,  a contradiction.
Consequently $G=\mathfrak{t}(G)$, i.e.  $G$ is a torsion group.

If $Syl_p(G)=\gp{1}$, then each subgroup  of $\mathfrak{t}(G)(=G)$ is abelian and normal in $G$ by Lemma \ref{L:1}.
Thus $FG$ is a direct sum of fields, so $N_*=\varnothing$, which is impossible.

Let   $P=Syl_p(G)\not=\gp{1}$.  Evidently    $\{h-1\mid h\in P\}$ is a finite subset of nilpotent elements,  so  $P$ is a finite subgroup in $G$. Assume that  $P$ is not normal in $G$. Since $P<V(FG)$ is Engel and nilpotent, some element of $P$ must have a conjugate  $w\in N_G(P)\setminus P$ (see \cite{Plotkin_book}, Lemma V.4.1.1, p.379 [p.307 in the English translation]). Consequently $P\lneqq\gp{P, w}$ and $\gp{P, w}$ is a finite $p$-group, a contradiction.

Hence $P$ is a  normal subgroup in $G$ and  the subset  of  nilpotent elements
\[
\{(h-1)g\lambda \mid h \in P,\;  g\in G,\; \lambda \in F \}
\]
is finite, so $G$ is  finite as well.
Since every finite Engel group is nilpotent by Zorn's theorem (see \cite{Robinson}, 12.3.4,  p.372),  $G$ is a finite nilpotent group which is a direct product of its Sylow subgroups (see  \cite{Hall_book}, 10.3.4, p.176) and  $G'$ is a finite $p$-group (see \cite{Bovdi_engel}, Theorem 3.2, p.175).
Moreover  any $q$-Sylow subgroup ($q\ne p$) is an abelian normal subgroup in $G$ and  each of  its subgroups is also normal in $G$ by  Lemma \ref{L:1}.
Consequently the finite group $G=P\times A$, where  $A=\cup_{q\not=p} Syl_q(G)$ is  an abelian normal subgroup of $G$.  Moreover, if $g\in G$ and  $z\in A$, then
\[
(z,g)=z^{-1}(g^{-1}zg)\in P\cap A=\gp{1}
\]
so $A$ is  central.  Since  $FG$ is a finite algebra,  $V(FG)$ is nilpotent (see \cite{Robinson}, 12.3.4,  p.372). \end{proof}

\smallskip

\begin{proof}[\underline{Proof of Theorem \ref{T:2}}]
Clearly $G'$ is locally nilpotent and  the ideal $\mathfrak I (P)$  is locally  nilpotent, hence $1+\mathfrak I (P)$ is a locally  nilpotent $p$-group which coincides  with the  $p$-Sylow normal  subgroup of $V(FG)$. In view of $G'\subseteq  P$ we obtain that
\begin{equation}\label{E:2}
V(FG)/\big(1+\mathfrak I (P)\big)\cong V(F [G/P]),\quad
V(FG)/\big(1+\mathfrak I (G')\big)\cong V(F [G/G'])
\end{equation}
are abelian groups.
From the second isomorphism in  (\ref{E:2}) it follows that $V'\leq 1+\mathfrak I (G')$.

Let $Syl_q(G)$ be a Sylow $q$-subgroup ($q\not=p$) of $G$.
According to Lemma \ref{L:1},    $Syl_q(G)$ is a  normal   abelian subgroup in $G$. Moreover, if $g\in G$ and  $z\in Syl_q(G)$, then
\[
(z,g)=z^{-1}(g^{-1}zg)\in P\cap Syl_q(G)=\gp{1}
\]
so $Syl_q(G)$ is  central. Consequently,   the group $\mathfrak{t}(G)$  is equal to $P\times D$
where $D=\times_{\substack{q\not=p}} Syl_q(G)\subseteq \zeta(G)$ and $\zeta(G)$ is the center of $G$.

Define the subgroup $M$ of $G/P$ to be the group generated by the central subgroup $DP/P$ and by all  elements of infinite order in $G/P$. Evidently, each  element of $M$ can be expressed as a product of an  element from  $DP/P(\cong D)$ and an element of infinite order. Consequently, the subgroup $M$ of $G/P$ has no element of order $p$ and $M/D$ is a free abelian group.

The algebra $F[M]$ can be expressed   as a crossed product $S(K, F[D])$ of the group $K\cong M/D$ over the group algebra $F[D]$ (see \cite{Bovdi_book}, Lemma 11.1, p.63).  Since $F[D]$  isomorphic to a central subalgebra of $S(K, F[D])$, the algebra $S(K, F[D])$  is a twisted group ring with $G$-basis
$\{t_g\mid g\in G\}$ such that $t_gt_h=t_{gh}\mu_{g,h}$ for all  $g,h\in G$, where   $\mu_{g,h}\in \mu=\{\mu_{a,b}\in U(FD)\mid a,b\in G\}$ is the factor system of $S(K, FD)$ (see \cite{Bovdi_book}, Chapter 11). Any unit $u\in FM$ can be written as $u=\sum_{i=1}^s\alpha_it_{g_i}\in S(K, FD)$,\quad  where $\alpha_i\in FD$. Define
\[
B=\gp{supp(\alpha_1), \ldots, supp(\alpha_s)}\leq D.
\]
If   $e_1,\ldots, e_n$ is  a complete set of primitive,  mutually
orthogonal idempotents of the f.d. commutative algebra $FB$ such that $e_1+ \cdots +e_n=1$,  then
\[
FB=FBe_1\oplus  \cdots \oplus FBe_n \quad \text{and} \quad FM=FMe_1\oplus \cdots \oplus FMe_n.
\]
The field $FBe_i$ is denoted   by $F_i$.  It is easy to see that
\[u\in S(K, FB)=S(K, F_1\oplus \cdots\oplus F_n)=S(K, F_1)\oplus \cdots \oplus S(K, F_n).
\]
Since  $K$ is an ordered group, all units in  $S(K, F_i)$  are trivial (see \cite{Bovdi_Crossed}, Lemma 3, p.495) and $ue_i=\beta_ie_it_{h_i}$ for some $h_i\in K$ and $\beta_i\in FB$. Consequently
\[
u=(\beta_1 e_1)t_{h_1}+(\beta_2 e_2) t_{h_2}+\cdots+(\alpha_n e_n) t_{h_n},\qquad  (\beta_i\in FB,\, h_i\in K)
\]
so  $U\big(S(K, F_i)\big)\cong U(F_i)\times K$.
Note that the description of the group $V(FM)$ depends only on the subgroup $D$.

Let us give all invariants of $V(FM)$ in the case when  the abelian group $D$ is finite. Set $B=D$.  Since  $U(S(K, F_i))\cong U(F_i)\times K$ for $i=1,\ldots, n$,
\[
\begin{split}
U(F[M])& \cong  U(S(K, F_1))\times  \cdots \times U(S(K, F_n))\\
& \cong  (U(F_1) \times K)\times \cdots \times (U(F_n)\times K)\\
& \cong  (U(F_1) \times \cdots \times U(F_n))\times L,\\
\end{split}
\]
where $L$ is a direct product of $n$ copies of $K$.
It follows that $U(FM)$ is an extension  of the torsion-free abelian group $L$   by a central subgroup  $U(FD)$. \end{proof}

\smallskip

\begin{proof}[\underline{Proof of Theorem \ref{T:3}}]
Since  $V(FG)$ is  a locally nilpotent group,   $G$  is also locally nilpotent.
Let $S=\gp{f_1,\ldots, f_s\mid f_i\in V(FG)}$ be  an f.~g.  subgroup of $V(FG)$.
Clearly
\[
H=\gp{supp(f_1),\ldots, supp(f_s)}
\]
is an f.~g. nilpotent  subgroup of $G$ and
$S\leq V(FH)< V(FG)$.
Hence  we may restrict our attention to the subgroup $V(FH)$, where  $H$ is an f.~g. nilpotent subgroup  of $G$.

Let $H$  be an f.~g. nilpotent group  containing a $p$-element and let $g, h\in H$ such that $(g,h)\not=1$. Obviously  $\mathfrak{t}(H)$ is finite (see   \cite{Hall_2}, 7.7, p. 29)  and the finite subgroup $Syl_p(\mathfrak{t}(H))$ is a direct factor of $\mathfrak{t}(H)$ (see  \cite{Hall_book}, 10.3.4, p.176). Consequently,  there exists  $c\in \zeta(H)$ of order $p$,   and $\widehat{c}=\sum_{i=0}^{p-1}c^i$   is a square-zero central element of $FG$. Since  $L=\gp{g,h, 1+g\widehat{c}}$ is an f.~g. nilpotent subgroup of $V(FG)$, so  $L$ is Engel and  there exists $m\in \mathbb{N}$ such that the nilpotency class $cl(L)$  of $L$ is at most   $p^m$.

Let us show that  $(g,h)$ is a $p$-element. Indeed, if  $q=p^m$ then  we have
\[
1=\big(1+g\widehat{c},h, q\big)= 1+\widehat{c}\sum_{i=0}^q (-1)^i\textstyle\binom {q}{i} g^{h^{q-i}}=1+\widehat{c}(g^{h^{q}}-g),
\]
because $\binom {q}{i} \equiv 0 \pmod{p}$ for  $0 <i< q$. It follows that  $\widehat{c}(g^{h^{q}}-g)=0$ and $(g,h^{q})\in supp(\widehat{c})$,
which  yields that  $g^{h^{q}}=c^ig$ for some $0\leq i<p$. Hence
\[
\big(\underbrace{h^{-q}\cdots h^{-q}}_p) g(\underbrace{h^{q}\cdots h^{q}}_p\big)=(c^i)^pg=g
\]
and $h^{p^{m+1}}\in C_G(g)$,  so  $h^{p^{m+1}}\in \zeta(L)$ for all $h\in L$ and  $L'$ is a finite $p$-group by a theorem of Schur (see \cite{Robinson}, 10.1.4,  p.287). Consequently  $G'$ is a $p$-group.

Conversely, it is sufficient  to prove  that if $H$ is an f.~g.  nilpotent subgroup of $G$ such that $H'$ is a $p$-group then   $V(FH)$ is nilpotent. We known that all subgroups and factor groups of $H$ are also f.~g.  groups
(\cite{Robinson}, 5.2.17, p.132). Moreover $\mathfrak{t}(H)=P\times D$ is finite (see   \cite{Robinson}, 12.1.1, p.356) and $H/\mathfrak{t}(H)$ is a direct product of a finite number of infinite cyclic groups.

Let $Syl_q(H)$ be a Sylow $q$-subgroup ($q\not=p$) of $H$. It is well known (see   \cite{Hall_2}, 7.7, p. 29  and  \cite{Hall_book}, 10.3.4, p.176) that $Syl_q(H)$ is normal in  $H$. Moreover, if $g\in H$ and  $z\in Syl_q(H)$, then
\[
(z,g)=z^{-1}(g^{-1}zg)\in P\cap Syl_q(G)=\gp{1}
\]
so $Syl_q(H)$ is  central. Consequently,   $\mathfrak{t}(H)=P\times D$,
where $D=\times_{\substack{q\not=p}} Syl_q(G)\subseteq \zeta(H)$ and   (\ref{E:1}) holds by Theorem \ref{T:4}({\rm v}).
Hence $V(F[G/P])$ is an extension  of an f.~g. abelian group $L$   by a central subgroup  $V(FD)$. By Lemma 2.4 from \cite{Bovdi_Khripta_2} (or by  Lemma 2.1 in \cite{Bovdi_engel}, on p.176) every extension of a nilpotent group $1+\mathfrak I (P)$ by an f.~g.  abelian group $L$ is nilpotent. Since $U(FD)$ is central so the group $U(FH)$ is nilpotent.\end{proof}

\smallskip

\begin{proof}[\underline{Proof of Theorem \ref{T:4}}]
If $V(FG)$ is locally nilpotent,  then it is Engel, so by Theorem 1.1 from \cite{Bovdi_engel},  the locally nilpotent group $G$ satisfies   conditions  (i)--(ii) of our theorem. Since a locally nilpotent group is an  u.~p. group,  the converse of our theorem follows from Theorem 2 in \cite{Khripta_2} and Theorem 1.1 in \cite{Bovdi_engel}.
\end{proof}

\end{document}